\documentclass[11pt,a4paper]{article}
\usepackage[english]{babel}
\usepackage[english]{layout}
\usepackage{amsmath,bbm,latexsym}
\usepackage{amsfonts,amssymb}
\usepackage{color,graphics,graphicx,amsthm}
\usepackage[T1]{fontenc}
\usepackage{color}
\def\eb{\begin{eqnarray*}}
\def\ee{\end{eqnarray*}}
\newcommand{\bc}{\begin{center}}
\newcommand{\ec}{\end{center}}

\newtheorem{thm}{Theorem}[section]
\newtheorem{defi}{Definition}[section]

\newtheorem{remark}{Remark}[section]

\def\non{\nonumber}

\def\mf{\mathfrak}
\newcommand{\al}{\alpha}
\newcommand{\be}{\beta}
\newcommand{\ga}{\gamma}
\newcommand{\de}{\delta}
\newcommand{\ep}{\varepsilon}
\newcommand{\et}{\eta}
\newcommand{\tf}{\vartheta}
\newcommand{\om}{\omega}
\begin{document}
\title{Algebraic Bethe ansatz for the elliptic quantum group $E_{\tau,\et}(A_2^{(2)})$}
\author{Nenad Manojlovi\'c\footnote{e-mail address: nmanoj@ualg.pt} and Zolt\'an Nagy\footnote{e-mail address: znagy@ualg.pt} \vspace{5pt}\\ Departamento de Matem\'atica \\FCT, Campus de Gambelas\\ Universidade do Algarve \\ 8005-139 Faro, Portugal}
\date{}

\maketitle
\begin{abstract}\noindent We implement the Bethe anstaz method for the elliptic quantum group $E_{\tau,\eta}(A_2^{(2)})$. The Bethe creation operators are constructed as polynomials of the Lax matrix elements expressed through a recurrence relation. We also give the
eigenvalues of the family of commuting transfer matrices defined in the tensor product of fundamental representations.
\end{abstract}
\section{Introduction}
Elliptic quantum groups are associative algebras related by Felder to elliptic solutions \cite{Ji} of the star-triangle
relation in statistical mechanics. Out of the Boltzmann weights of the corresponding interaction-round-a-face (IRF)
 model, one builds a dynamical $R$-matrix which is a solution of the dynamical Yang-Baxter
 equation, a deformation of the usual Yang-Baxter equation. Many of the concepts and methods one is
 familiar with in the field of Quantum Inverse Scattering Method \cite{Fa,QISM,Ku} can be applied in the context
 of elliptic quantum groups. For example a family of commuting operators (transfer matrix) can be associated to
 every representation of the algebra, and a variant of the algebraic Bethe ansatz method can be implemented
 to construct common eigenvectors of these families of operators.

The transfer matrix in a multiple tensor product of the so called
 fundamental representation
 can be identified to the row-to-row transfer matrix of the original IRF model; whereas for certain
 highest weight representations one can derive from the transfer matrix the Hamiltonian of the
corresponding Ruijsenaars-Schneider model with special integer coupling constants \cite{FeVa3,Sa}. The corresponding eigenvalue problem
can be viewed as the eigenvalue problem of the $q$-deformed Lam\'e equation\cite{FeVa}. The quasiclassical
limit of this construction leads to Calogero-Moser Hamiltonians: scalar or spin type, depending on the representation
chosen \cite{ABB}.

In this article we present the algebraic Bethe ansatz for the elliptic quantum group $E_{\tau,\et}(A_2^{(2)})$ \cite{Ko}.
The method is very similar to that described in \cite{FeVa,MaNa2} in that the main difficulty is the
definition of the Bethe state creation operator which becomes a complicated polynomial of the algebra generators.
We give the expression of this polynomial as a recurrence relation and derive the Bethe equations
in the simplest representation of the algebra.
\section{Representations of the elliptic quantum group $E_{\tau,\eta}(A_2^{(2)})$}

Following Felder \cite{Fe} we associate a dynamical $R$-matrix to the elliptic solution of the star-triangle
relation given by Kuniba \cite{Kun}. This $R$-matrix has a remarkably similar structure to the $B_1$ type
matrix \cite{MaNa2}, but its entries are defined in terms of two different theta functions instead of
just one. To write down the $R$-matrix, we first fix two complex parameters $\eta,\tau$ such that $\mbox{Im}(\tau)>0$.
We use the following definitions of Jacobi's theta functions with the elliptic nome set to: $p=e^{2 i \pi \tau}$.
\eb
\tf(u,p)=\theta_1(\pi u)=2 p^{1/8} \sin(\pi u) \prod_{j=1}^{\infty} (1-2p^j \cos(2 \pi u)+ p^{2j})(1-p^j)\\
\tf_v(u,p)=\theta_4(\pi u)=\prod_{j=1}^{\infty} (1-2p^{j-1/2}\cos(2\pi u)+p^{2j-1})(1-p^j)
\ee
We only write the explicit nome dependence if it is different from $p$

These functions verify the following quasiperiodicity properties:
\eb
\tf(u+1)=-\tf(u); \ \tf(u+\tau)= -e^{-i \pi \tau-2i \pi \tau}\tf(u) \\
\tf_v(u+1)=\tf_v(u); \ \tf_v(u+\tau)=-e^{-i \pi \tau-2i \pi \tau}\tf_v(u)
\ee
For the sake of completenes, we display additional useful identities:
\eb
\tf_v(u)&=&i e^{-i\pi u+i \pi \tau/4}\tf(u-\tau/2)\\
\frac{\tf_v(2u_1,p^2)}{\tf_v(2u_2,p^2)}&=&\frac{\tf(u_1-\tau/2)\tf(u_1+1/2-\tau/2)}{\tf(u_2-\tau/2)\tf(u_2+1/2-\tau/2)} \ e^{-i\pi (u_1-u_2)}
\ee
which will allow eventually to reduce the matrix entries to a functional form containing only one theta function.

We define the following functions.
\begin{eqnarray*}
g(u) &=& \frac{\tf(3\eta +1/2 -u)\tf(u-2\eta)}{\tf(3\eta+1/2)\tf(-2\eta)}\\
\al(q_1,q_2,u) & =& \frac{\tf(3\eta+1/2-u)\tf(q_{12}-u)}{\tf(3\eta+1/2)\tf(q_{12})}\\
\be(q_1,q_2,u) &=& \frac{\tf(3\et+1/2-u)\tf(u)}{\tf(-2\et)\tf(3\et+1/2)}
\left(\frac{\tf(q_{12}-2\et)\tf(q_{12}+2\et)}{\tf(q_{12})^2}\right)^{1/2}\\
\ga(q_1,q_2,u)&= & \frac{\tf(u)\tf(q_1+q_2+\et+1/2-u)}{\tf(3\et+1/2)\tf(q_1+q_2-2\et)}\sqrt{G(q_1)G(q_2)}\\
\de(q,u) & = & \frac{\tf(3\et+1/2-u)\tf(2q-2\et-u)}{\tf(3\et+1/2)\tf(2q-2\et)}+\frac{\tf(u)\tf(2q+\et+1/2-u)}
{\tf(3\et+1/2)\tf(2q-2\et)}G(q)\\
\ep(q,u)&=& \frac{\tf(3\et+1/2+u)\tf(6\et-u)}{\tf(3\et+1/2)\tf(6\et)}-\frac{\tf(u)\tf(3\et+1/2-u)}{\tf(3\et+1/2)\tf(6\et)}\times\\
&&\left( \frac{\tf(q+5\et)}{\tf(q-\et)}G(q)+\frac{\tf(q-5\et)}{\tf(q+\et)}G(-q)\right)
\end{eqnarray*}
where
\begin{equation*}
G(q)=\left\{ \begin{array}{ll}1 & \mbox{if $q=\et$}\\
\frac{\tf(q-2\et)\tf_v(2q-4\et,p^2)}{\tf(q)\tf_v(2q,p^2)} & \mbox{otherwise}
 \end{array}
 \right.
\end{equation*}

Let $V$ be a three dimensional complex vector space, identified with $\mathbb{C}^3$, with the standard basis
$\{e_1,e_2,e_3\}$. The elementary operators are defined by: $E_{ij}e_k=\delta_{jk} e_i$ and let $h=E_{11}-E_{33}$.

The R-matrix then has the form.
\eb\label{Rmat}
R(q,u)&=&g(u)E_{11}\otimes E_{11}+g(u)E_{33}\otimes E_{33}+\ep(q,u)E_{22}\otimes E_{22}\\
&+&\al(\eta,q,u)E_{12}\otimes E_{21}+\al(q,\eta,u)E_{21}\otimes E_{12}+\al(-q,\eta,u)E_{23}\otimes E_{32}\\
&+& \al(\eta,-q,u)E_{32}\otimes E_{23}\\
&+& \be(\eta,q,u) E_{22}\otimes E_{11}+\be(q,\eta,u) E_{11}\otimes E_{22}+\be(-q,\eta,u) E_{33}\otimes E_{22}\\
&+& \be(\eta,-q,u)E_{22}\otimes E_{33}\\
&+&\ga(-q,q,u)E_{33}\otimes E_{11}+\ga(-q,\eta,u)E_{23} \otimes E_{21}+ \ga(\eta,q,u) E_{32} \otimes E_{12}\\
&+& \ga(q,-q,u) E_{11} \otimes E_{33}+ \ga(q,\eta,u) E_{21} \otimes E_{23}+ \ga(\eta,-q,u) E_{12} \otimes E_{32}\\
&+& \de(q,u) E_{31}\otimes E_{13}+\de(-q,u) E_{13} \otimes E_{31}
\ee
\begin{remark}
By taking first the trigonometric limit ($p\rightarrow 0$) and then the nondynamical limit ( $q\rightarrow \infty$)
one recovers, up to normalization, the vertex type $R$-matrix given in \cite{Ji}.
\end{remark}
This $R$-matrix also enjoys the unitarity property:
\begin{eqnarray}\label{unit}
R_{12}(q,u)R_{21}(q,-u)=g(u)g(-u)\mathbbm{1}
\end{eqnarray}
and it is of zero weight:
\eb
\left[h \otimes \mathbbm{1}+\mathbbm{1} \otimes h,R_{12}(q,u)\right]=0 \qquad (h \in \mf{h})
\ee

The $R$-matrix also obeys the dynamical quantum Yang-Baxter equation (DYBE) in $End(V \otimes V \otimes V)$:
\eb
&&R_{12}(q-2\eta h_3,u_{12}) R_{13}(q,u_1) R_{23}(q-2\eta h_1,u_2)=\\
 && R_{23}(q,u_2)R_{13}(q-2\eta h_2,u_1)
R_{12}(q,u_{12})
\ee
where the "dynamical shift" notation has the usual meaning:
\begin{eqnarray}\label{shift}
R_{12}(q-2\eta h_3,u) \cdot v_1\otimes v_2 \otimes v_3 = \left(R_{12}(q-2\eta \lambda,u) v_1\otimes v_2\right)
\otimes v_3
\end{eqnarray}
whenever $h v_3= \lambda v_3$.
This definition of the dynamical shift can be extended to more general situations \cite{Fe}. Indeed, let
the one dimensional Lie algebra $\mathfrak{h}=\mathbb{C}h$ act on
$V_1, \ldots, V_n$ in such a way that
each $V_i$ is a direct sum of (finite dimensional) weight subspaces $V_i[\lambda]$ where $h\cdot x=\lambda x$ whenever
$x\in V_i[\lambda]$. Such module spaces ar called diagonlizable $\mathfrak{h}$-modules.
Let us denote by $h_i \in \mbox{End}(V_1\otimes \ldots \otimes V_n)$ the
operator $\ldots \otimes \mathbbm{1} \otimes h\otimes \mathbbm{1}\otimes \ldots$ acting non trivially
only on the $i$th factor. Now let $f(q) \in \mbox{End} (V_1\otimes \ldots \otimes V_n)$ be a function on $\mathbb{C}$.
Then $f(h_i) x=f(\lambda) x$ if $h_i \cdot x=\lambda x$.

Now we describe the notion of representation of (or module over) $E_{\tau,\eta}(A_2^{(2)})$. It
 is a pair $(\mathcal{L}(q,u),W)$ where $W=\oplus_{\lambda \in \mathbb{C}}W[\lambda]$ is a diagonalizable $\mf{h}$-module,
 and $\mathcal{L}(q,u)$ is an operator in $\mathrm{End}(V \otimes W)$ obeying:
\eb\label{RLL}
&&R_{12}(q-2\eta h_3,u_{12}) \mathcal{L}_{13}(q,u_1) \mathcal{L}_{23}(q-2\eta h_1,u_2)=\\
 && \mathcal{L}_{23}(q,u_2)\mathcal{L}_{13}(q-2\eta h_2,u_1)
R_{12}(q,u_{12})
\ee

$\mathcal{L}(q,u)$ is also of zero weight
\eb
\left[h_V \otimes \mathbbm{1}+\mathbbm{1} \otimes h_W , \mathcal{L}(q,u)\right]=0 \qquad (h \in \mf{h})
\ee
where the subscripts remind the careful reader that in this formula $h$ might act in a different way on spaces
$W$ and $V$.

An example is given immediately by $W=V$ and $\mathcal{L}(q,u)=R(q,u-z)$ which is called the fundamental
representation with evaluation point $z$ and is denoted by $V(z)$.
A tensor product of representations can also be defined which corresponds to the existence of a coproduct-like
structure at the abstract algebraic level. Let $(\mathcal{L}(q,u),X)$ and $(\mathcal{L}'(q,u),Y)$
be two $E_{\tau,\eta}(A_2^{(2)})$
modules, then $(\mathcal{L}_{1X}(q-2\eta h_Y,u)\mathcal{L}'_{1Y}(q,u),X\otimes Y)$ is a
representation of $E_{\tau,\eta}(A_2^{(2)})$ on
$X \otimes Y$ endowed, of course, with the tensor product $\mf{h}$-module structure.

The operator $\mathcal{L}$ is reminiscent of the quantum Lax matrix in the FRT formulation
of the quantum inverse scattering
method, although it obeys a different exchange relation,
therefore we will also call it a Lax matrix. This allows us to view
the $\mathcal{L}$ as a matrix with operator-valued entries.

Inspired by that interpretation, for any module over $E_{\tau,\eta}(A_2^{(2)})$ we define the corresponding
\textbf{operator algebra} of finite difference operators following the method in \cite{Fe2}.
Let us take an arbitrary representation $\mathcal{L}(q,u) \in \mathrm{End}(V \otimes W)$.
The elements of the operator algebra corresponding to this representation will act on the space $\mathrm{Fun}(W)$ of
meromorphic functions of $q$ with values in $W$. Namely let $L \in \mathrm{End}(V \otimes \mathrm{Fun}(W))$
be the operator defined as:
\begin{eqnarray}\label{Lti}
L(u)=\left( \begin{array}{ccc}
A_1(u)& B_1(u)& B_2(u)\\
C_1(u) & A_2(u) & B_3(u)\\
C_2(u) & C_3(u) &A_3(u)
 \end{array}\right)=\mathcal{L}(q,u)e^{-2\eta h \partial_q}
\end{eqnarray}
We can view it as a matrix with entries in $\mathrm{End}(\mathrm{Fun}(W))$:
It follows from equation (\ref{RLL}) that $L$
verifies:
\begin{eqnarray}\label{RLLti}
R_{12}(q-2\eta h,u_{12}) \ L_{1W}(q,u_1) L_{2W}(q,u_2)= L_{2W}(q,u_2)L_{1W}(q,u_1) \
\tilde{R}_{12}(q,u_{12})
\end{eqnarray}
with $\tilde{R}_{12}(q,u):= \exp(2\eta(h_1+h_2)\partial_q)R_{12}(q,u)\exp(-2\eta(h_1+h_2)\partial_q)$

The zero weight condition on $L$ yields the relations:
\eb\label{order}
\left[h,A_i\right]=0 ; \ \ \left[h,B_j\right]=-B_j \quad (j=1,3), \ \left[h,B_2\right]=-2B_2\\
\left[h,C_j\right]=C_j \quad (j=1,3), \ \left[h,C_2\right]=2C_2
\ee
so $B_i$'s act as lowering and $C_i$'s as raising operators with respect to the $h$-weight.
From the definition \eqref{Lti} one can derive the action of the operator algebra generators on functions:
\eb
A_1(u) f(q)= f(q-2\eta)A_1(u);\ B_1(u)f(q)= f(q)B_1(u); \\
B_2(u)f(q)= f(q+2\eta)B_2(u)
\ee
and analogously for the other generators.

Finally the following theorem shows how to associate a family of commuting quantities to a representation
of the elliptic quantum group
\begin{thm}
Let $W$ be a representation of $E_{\tau,\eta}(A_2^{(2)})$. Then the transfer matrix defined by $t(u)=Tr L(u) \in
\mathrm{End}(\mathrm{Fun}(W))$
preserves the subspace $\mathrm{Fun}(W)[0]$ of functions with values in the zero weight subspace of $W$.
When restricted to this subspace, they commute at different values of the spectral parameter:
\eb
\left[t(u),t(v)\right]=0
\ee
\end{thm}
\begin{proof}
The proof is analogous to references \cite{ABB,FeVa3}
\end{proof}

\section{Bethe ansatz}
Algebraic Bethe ansatz techniques can be applied to the diagonalization of transfer matrices defined on a highest
weight module.
In this section, analogously to \cite{MaNa2}, we choose to work with the module $W=V(z_1)\otimes \ldots \otimes V(z_n)$ which
has a highest weight $|0\rangle =e_1\otimes \ldots \otimes e_1 \in \textrm{Fun}(W)[n]$. Any non-zero highest weight
vector
$|\Omega\rangle$ is of the form $|\Omega \rangle = f(q) |0 \rangle$ with a suitably chosen $f(q)$.
We have indeed:
 \eb
 C_i(u)|\Omega \rangle=0 \qquad (i=1,2,3)
 \ee
showing that $|\Omega \rangle$ is a highest weight vector; it is of $h$-weight $n$.

\eb
A_1(u)|\Omega\rangle=a_1(u)\frac{f(q-2\eta)}{f(q)}|\Omega\rangle\\
\quad A_2(u)|\Omega\rangle=a_2(q,u)|\Omega\rangle \quad A_3(u)|\Omega\rangle=a_3(q,u)\frac{f(q+2\eta)}{f(q)}|\Omega\rangle
\ee
with the eigenvalues:
\eb
a_1(u)&=&\prod_{i=1}^n \frac{\tf(3\et+1/2-u+z_i)\tf(u-z_i+2\et)}{\tf(3\et+1/2)\tf(-2\et)}\\
a_2(q,u)&=&\prod_{i=1}^n \frac{\tf(3\et+1/2-u+z_i)\tf(u-z_i)}{\tf(-2\et)\tf(3\et+1/2)}\times
\left(\frac{\tf(q+\et)\tf(q-2\et n-\et)}{\tf(q-\et)\tf(q-2\et n+\et)}\right)^{\frac{1}{2}}\\
a_3(q,u)&=& \prod_{i=1}^n \frac{\tf(u-z_i)\tf(\et+1/2-u+z_i)}{\tf(3\et+1/2)\tf(-2\et)}\times\\
&&\left(\frac{\tf(q-2\et n)\tf(q+2\et)\tf_v(2q-4\et n,p^2)\tf_v(2q+4\et,p^2)}
{\tf(q)\tf(q-2\et n+2\et)\tf_v(2q,p^2)\tf_v(2q-4\et n+4\et,p^2)}\right)^{\frac{1}{2}}
\ee

We look for the eigenvectors of the transfer matrix $t(u)=Tr L(u)|_{\mathrm{Fun}(W)[0]}$ in the form
$\Phi_n(u_1,\ldots,u_n) |\Omega \rangle$ where $\Phi_n(u_1,\ldots,u_n)$ is a polynomial of the Lax matrix elements
lowering the $h$-weight by $n$.

During the calculations, we need the commutation relations of the generators of the algebra. These relations
can be derived from \eqref{RLLti} and we only list some of the relations to introduce further notation:
\eb
B_1(u_1)B_1(u_2)&=&\omega_{21}\left(B_1(u_2)B_1(u_1)-\frac{1}{y_{21}(q)}B_2(u_2)A_1(u_1)\right)+
\frac{1}{y_{12}(q)}B_2(u_1)A_1(u_2) \label{crB1B1} \\
A_1(u_1)B_1(u_2)&=&z_{21}(q)B_1(u_2)A_1(u_1)-\frac{\al_{21}(\eta,q)}{\be_{21}(\eta,q)}B_1(u_1)A_1(u_2) \\
A_1(u_1)B_2(u_2)&=&\frac{1}{\ga_{21}(-q,q)}\left( g_{21}B_2(u_2)A_1(u_2)+\ga_{21}(-q,\eta)B_1(u_1)B_1(u_2) \non\right. \\
&&\left.-\de_{21}(-q)B_2(u_1)A_1(u_1)\right) \\
 B_1(u_2)B_2(u_1)&=&\frac{1}{g_{21}}\left( \be_{21}(\eta,-q)B_2(u_1)B_1(u_2)+\al_{21}(\eta,-q)B_1(u_1)B_2(u_2)\right)
 \\
 B_2(u_2)B_1(u_1)&=&\frac{1}{g_{21}}\left(- \be_{21}(-q,\eta)B_1(u_1)B_2(u_2)+\al_{21}(-q,\eta)B_2(u_1)B_1(u_2)\right)\label{crB2B1}
\ee
where
\begin{gather*}
y(q,u)=\frac{\ga(-q,q,u)}{\ga(\eta,q,u)} \nonumber\\
z(q,u)=\frac{g(u)}{\be(\eta,q,u)}\nonumber
\end{gather*}
and
\eb
\omega(q,u)=\frac{g(u)\ga(q,-q,u)}{\ep(q,u)\ga(q,-q,u)-\ga(q,\eta,u)\ga(\eta,-q,u)}
\ee
This function turns out to be independent of $q$ and takes the following simple form:
\eb
\om(u)=\frac{\tf(u+1/2-\eta)}{\tf(u+1/2+\eta)}=\frac{1}{\om(-u)}
\ee
This equality can be verified by looking at the quasiperiodicity properties and poles of both sides.

Following \cite{MaNa,MaNa2} and \cite{Ta} we define the creation operator $\Phi_m$ by a recurrence relation.

\begin{defi}
Let $\Phi_m$ be defined by the recurrence relation for $m\geq 2$:
\eb
&&\Phi_m(u_1,\ldots,u_m)=B_1(u_1)\Phi_{m-1}(u_2,\ldots, u_m)\\
&&-\sum_{j=2}^m\frac{\prod_{k=2}^{j-1}\omega_{jk}}{y_{1j}(q)}
\prod^m_{\stackrel{k=2}{ k \neq j}} z_{kj}(q+2\eta)\ B_2(u_1) \Phi_{m-2}(u_2,\ldots,\widehat{u_j},\ldots,u_m)A_1(u_j)
\ee
where $\Phi_0=1; \ \Phi_1(u_1)=B_1(u_1)$ and the parameter under the hat is omitted.
\end{defi}

For general $m$ we prove the following theorem.
\begin{thm}
$\Phi_m$ verifies the following symmetry property:
\begin{equation}\label{symm}
\Phi_m(u_1,\ldots,u_m)=\omega_{i+1,i}\Phi_m(u_1,\ldots,u_{i-1},u_{i+1},u_i,u_{i+2},\ldots,u_m)
\qquad (i=1,2,\ldots,m-1).
\end{equation}
\end{thm}
\begin{proof}
The proof is analogous to that in \cite{MaNa} and is by induction on $m$. It is straightforward for $i\neq 1$. For $i=1$ one has to expand $\Phi_m$ one step further and then substitute it into \eqref{symm}. The right hand side is then brought to normal order of the spectral parameters using the relations between Lax matrix entries. The equality \eqref{symm} then holds thanks to the following identitites verified by the $R$-matrix elements:
\begin{gather*}\label{A1.4}
-\frac{\omega_{12}g_{21}}{y_{23}(q)\be_{21}(-q,\eta)}+\frac{\al_{21}(\eta,-q)}{\be_{21}(-q,\eta)y_{13}(q)}=
-\frac{\omega_{31}z_{13}(q+2\eta)}{y_{23}(q)}-\frac{\al_{31}(\eta,q+2\eta)}{\be_{31}(\eta,q+2\eta)y_{21}(q)}
\end{gather*}
and
\begin{gather*}
\omega_{12}\left(\frac{\omega_{42}z_{24}(q+2\eta)z_{34}(q+2\eta)}{y_{14}(q)y_{23}(q+2\eta)}+
\omega_{34}\frac{\omega_{32}z_{23}(q+2\eta)z_{43}(q+2\eta)}{y_{13}(q)y_{24}(q+2\eta)} \right) \non\\
-\left( \frac{\om_{41}z_{14}(q+2\eta)z_{34}(q+2\eta)}{y_{24}(q)y_{13}(q+2\eta)}+\frac{\om_{34}\om_{31}
z_{13}(q+2\eta)z_{43}(q+2\eta)}{y_{23}(q)y_{14}(q+2\eta)}\right)\non \\
+\frac{\om_{12}}{y_{12}(q)}\left( \frac{\de_{42}(-q-2\eta)}{\ga_{42}(-q-2\eta,q+2\eta)y_{43}(q)}+
\frac{z_{42}(q+2\eta)\al_{32}(\eta,q+2\eta)\om_{24}}{\be_{32}(\eta,q+2\eta)y_{24}(q+2\eta)} \right) \non\\
-\frac{1}{y_{21}(q)}\left(\frac{\de_{41}(-q-2\eta)}{\ga_{41}(-q-2\eta,q+2\eta)y_{43}(q)}+\frac{z_{41}(q+2\eta)
\al_{31}(\eta,q+2\eta)\om_{14}}{\be_{31}(\et,q+2\eta)y_{14}(q+2\eta)}
\right)=0
\end{gather*}

\end{proof}

The next step in the application of the Bethe ansatz scheme is the calculation of the action of the transfer
matrix on the Bethe vector. For the highest weight module $W$ described in the beginning of this section one
has to choose the $n$-th order polynomial $\Phi_n$ for the creation operator to reach the zero weight
subspace of $W$. The action of the transfer matrix on this state will yield
three kinds of terms. The first part (usually called wanted terms in the literature)
will tell us the eigenvalue of the transfer matrix, the second part (called unwanted terms) must be annihilated
by a careful choice of the spectral parameters $u_i$ in $\Phi_n(u_1,\ldots,u_n)$; the vanishing of these unwanted
terms is ensured if the $u_i$ are solutions to the so called Bethe equations. The third part contains terms
ending with a raising operator acting on the pseudovacuum and thus vanishes.

The action of $A_1(u)$ on $\Phi_n$ is given by
\begin{eqnarray}\label{A1phi}
A_1(u)\Phi_n &=& \prod_{k=1}^n z_{ku}(q) \Phi_n A_1(u)+\\
&&\sum_{j=1}^n D_j \prod_{k=1}^{j-1}\omega_{jk} B_1(u)\Phi_{n-1}(u_1,\hat{u_j},u_n)A_1(u_j)+ \non \\
&&\sum_{l<j}^n E_{lj}\prod_{k=1}^{l-1} \omega_{lk} \non
\prod_{\stackrel{k=1}{k\neq l}}^{j-1}\omega_{jk} B_2(u) \Phi_{n-2}(u_1,\hat{u_l},\hat{u_j},u_n)A_1(u_l)A_1(u_j)
\end{eqnarray}
To calculate the first coefficients we expand $\Phi_n$ with the help
of the recurrence relation, then use the commutation relations to push $A_1(u_1)$ to the right.
This yields:
\eb
D_1&=&\frac{\al_{1u}(\eta,q)}{\be_{1u}(\eta,q)}\prod_{k=2}^n
z_{k1}(q)\\
E_{12}&=&\left( \frac{\de_{1u}(-q)}{\ga_{1u}(-q,q)y_{12}(q-2\eta)}+
\frac{z_{1u}(q)\al_{2u}(\eta,q)\omega_{u 1}}{\be_{2u}(\eta,q)y_{u 1}(q)}\right)
\prod_{k=3}^n z_{k1}(q+2\eta)z_{k2}(q)
\ee
The direct calculation of the remaining coefficients is less straightforward. However, the symmetry of the left
hand side of \eqref{A1phi} implies that $D_j$ for $j\geq 1$ can be obtained by substitution
$u_1 \rightsquigarrow u_j$ in $D_1$ and $E_{lj}$ by the substitution
$u_1 \rightsquigarrow u_l$, $ u_2 \rightsquigarrow u_j$

The action of $A_2(u)$ and $A_3(u)$ on $\Phi_n$ will yield also terms ending in $C_i(u)$'s.

The action of $A_2(u)$ on $\Phi_n$ will have the following structure.
\eb
A_2(u)\Phi_n &=& \prod_{k=1}^n \frac{z_{u k}(q-2\eta(k-1))}{\omega_{u k}} \Phi_n A_2(u)+\\
&&\sum_{j=1}^n F^{(1)}_j \prod_{k=1}^{j-1}\omega_{jk} B_1(u)\Phi_{n-1}(u_1,\hat{u_j},u_n)A_2(u_j)+\\
&&\sum_{j=1}^n F^{(2)}_j \prod_{k=1}^{j-1}\omega_{jk} B_3(u)\Phi_{n-1}(u_1,\hat{u_j},u_n)A_1(u_j)+\\
&&\sum_{l<j}^n G^{(1)}_{lj} \prod_{k=1}^{l-1}\omega_{lk}\prod_{\stackrel{k=1}{k\neq l}}^{j-1} \omega_{jk}
B_2(u) \Phi_{n-2}(u_1,\hat{u_l},\hat{u_j},u_n) A_1(u_l)A_2(u_j)+\\
&&\sum_{l<j}^n G^{(2)}_{lj} \prod_{k=1}^{l-1}\omega_{lk}\prod_{\stackrel{k=1}{k\neq l}}^{j-1} \omega_{jk}
B_2(u)\Phi_{n-2}(u_1,\hat{u_l},\hat{u_j},u_n)A_1(u_j)A_2(u_l)+\\
&&\sum_{l<j}^n G^{(3)}_{lj}\prod_{k=1}^{l-1}\omega_{lk}\prod_{\stackrel{k=1}{k\neq l}}^{j-1} \omega_{jk}
B_2(u)\Phi_{n-2}(u_1,\hat{u_l},\hat{u_j},u_n)A_2(u_l)A_1(u_j)+\\
&&\textit{terms ending in C}
\ee

We give the coefficients $F^{(k)}_1$ and $G^{(k)}_{12}$, the remaining ones are obtained by the same
substitution as for $A_1(u)$

\eb
F^{(1)}_1&=&-{\frac{\al_{u 1}(q,\eta)}{\be_{u 1}(\eta,q)}\prod_{k=2}^n
\frac{z_{1k}(q-2\eta(k-1))}{\omega_{1k}}}\\
F^{(2)}_1&=&\frac{1}{y_{u 1}(q)}\prod_{k=2}^n z_{k1}(q+2\eta)\\
G^{(1)}_{12}&=& \frac{1}{y_{u 1}(q)}\left( \frac{z_{u 1}(q)\al_{u 2}(q-2\eta,\eta)}{\be_{u 2}(\eta,q-2\eta)}-
\frac{\al_{u 1}(q,\eta)\al_{12}(q-2\eta,\eta)}{\be_{u 1}(\eta,q)\be_{12}(\eta,q-2\eta)}\right)
\prod_{k=3}^n \frac{z_{k1}(q+2\eta)z_{2k}(q-2\eta(k-1))}{\omega_{2k}}\\
G^{(2)}_{12}&=& \frac{\al_{u 1}(q,\eta)\al_{12}(q-2\eta,\eta)}{\be_{u 1}(\eta,q) y_{u 1}(q)\be_{12}(\eta,q-2\eta)}
\prod_{k=3}^n \frac{z_{k2}(q+2\eta)z_{1k}(q-2\eta(k-1))}{\omega_{1k}}\\
G^{(3)}_{12}&=&-\frac{\al_{u 1}(q,\eta)}{\be_{u 1}(-q,\eta)}\left( \frac{z_{u 1}(q)}{\omega_{u 1}y_{u 2}(q)}-
\frac{\al_{u 1}(\eta,-q)}{y_{12}(q)\be_{u 1}(\eta,q)}\right)
\prod_{k=3}^n \frac{z_{k2}(q+2\eta)z_{1k}(q-2\eta(k-2))}{\omega_{1k}}
\ee
It is instructing to give explicitly the expression of $F^{(1)}_l$
\eb
F^{(1)}_l=-\frac{\al_{ul}(q,\eta)}{\be_{ul}(\eta,q)}\times \left(\frac{\tf(q-3\eta)\tf(q-2\eta n +\eta)}
{\tf(q-\eta)\tf(q-2\eta n -\eta)}
\right)^{\frac{1}{2}}
\prod_{\stackrel{k=1}{k\neq l}}^{n} \frac{\tf(u_{1k}-2\eta)\tf(u_{1k}+1/2+\eta)}{\tf(u_{1k}+1/2-\eta)\tf(u_{1k})}
\ee
The action of $A_3(u)$ on the Bethe vector is somewhat simpler.
\eb
A_3(u) \Phi_n &=& \prod_{k=1}^n -\frac{\be_{u k}(-q,\eta)}{\ga_{u k}(-q+2\eta(k-1),-)}
\Phi_n A_3(u)+ \\
&&\sum_{j=1}^n H_j \prod_{k=1}^{j-1}\omega_{jk} B_3(u)\Phi_{n-1}(u_1,\hat{u_j},u_n)A_2(u_j)+\\
&&\sum_{l<j}^n I_{lj}\prod_{k=1}^{l-1}\omega_{lk}\prod_{\stackrel{k=1}{k\neq l}}^{j-1} \omega_{jk}
B_2(u) \Phi_{n-2}(u_1,\hat{u_l},\hat{u_j},u_n) A_2(u_l)A_2(u_j)+\\
&&\textit{terms ending in C}
\ee
where to save space used the notation $\ga_{uk}(x,-)=\ga_{uk}(x,-x)$.
We give the coefficients $H_1$ and $I_{12}$, the rest can be obtained by the substitution of
the spectral parameters as before.
\eb
H_1&=&-\frac{1}{y_{u 1}(q)} \prod_{k=2}
\frac{z_{1k}(q-2\eta(k-2))}{\omega_{1k}}\\
I_{12}&=&\frac{1}{\ga_{u 2}(-q,q)}\left( \frac{\de_{u 2}(q)}{y_{12}(q-2\eta)}-\frac{\al_{u 1}(q,\eta)}
{y_{u 2}(q-2\eta)}\right)
 \prod_{k=3}\frac{z_{2u}(q-2\eta (k-2))z_{1u}(q-2\eta (k-2))}{\omega_{1k}\omega_{2k}}
\ee

We are now going to gather the similar terms together and find a sufficient condition for the
cancelation of the unwanted terms.
We write the action of the transfer matrix in the following regrouped form:
\eb
&&t(u)\Phi_n |\Omega \rangle = \Lambda \Phi_n |\Omega\rangle +\\
&& \sum_{j=1}^n K^{(1)}_j \prod_{k=1}^{j-1}\omega_{jk} B_1(u) \Phi_{n-1}(u_1,\hat{u_j},u_n)|\Omega\rangle+\\
&& \sum_{l<j}^n K^{(2)}_{lj} \prod_{k=1}^{l-1}\omega_{lk}\prod_{\stackrel{k=1}{k\neq l}}^{j-1} \omega_{jk}
B_2(u) \Phi_{n-2}(u_1,\hat{u_l},\hat{u_j},u_n)|\Omega\rangle+\\
&& \sum_{j=1}^n K^{(3)}_j \prod_{k=1}^{j-1}\omega_{jk} B_3(u) \Phi_{n-1}(u_1,\hat{u_j},u_n)|\Omega\rangle
\ee
The eigenvalue is written in a general form as:
\eb
\Lambda(u,\{u_j\})&=&\prod_{k=1}^n z_{ku}(q)\times a_1(q,u)\frac{f(q-2\eta)}{f(q)}+\prod_{k=1}^n
\frac{z_{u k}(q-2\eta(k-1))}
{\omega_{u k}}\times a_2(q,u)+\\
&&\prod_{k=1}^n \frac{\be_{u k}(-q,\eta)}{\ga_{u k}(-q+2\eta(k-1),-)}\times a_3(q,u)\frac{f(q+2\eta)}{f(q)} \ .
\ee
where $f(q)$ will be fixed later so as to eliminate to $q$-dependence.

The condition of cancelation is then $K^{(1)}_j=K^{(3)}_j=0 \textrm{ for } 1 \leq j $ and
$K^{(2)}_{lj}=0 \textrm{ for } 1\leq l \leq j$ with the
additional requirement that these three different kinds of condition should in fact lead to the same set
of $n$ nonlinear Bethe equations fixing the $n$ parameters of $\Phi_n$.

Let us first consider the coefficient $K^{(1)}_1$:
\eb
K^{(1)}_1=D_1a_1(u_1)\frac{f(q-2\eta)}{f(q)}+F^{(1)}_1 a_2(q,u_1)
\ee
The condition $K^{(1)}_1=0$ is then equivalent to:
\begin{eqnarray}\label{Bethe1}
\frac{a_1(u_1)}{a_2(q,u_1)}&=&\frac{f(q)}{f(q-2\et)}\left( \frac{\tf(q-2\et n+\et)}{\tf(q-2\et n-\et)}\right)^{1/2}
\frac{\tf(q-3\et)^{n/2}\tf(q+\et)^{\frac{n-1}{2}}}{\tf(q-\et)^{n-1/2}} \times \nonumber\\
&&\prod_{k=2}^n \frac{\tf(u_{1k}-2\et)\tf(u_{1k}+1/2+\et)}{\tf(u_{1k}+2\et)\tf(u_{1k}+1/2-\et)}
\end{eqnarray}
Now one has to check that the remaining two conditions lead to the same Bethe equations.
The condition
\eb
0=K^{(3)}_1=F^{(2)}_1 a_1(u_1)\frac{f(q)}{f(q+2\eta)}+H_1a_2(q+2\eta)
\ee
yields the same Bethe equation as in \eqref{Bethe1} thanks to the identity (from the unitarity condition \eqref{unit}):
\eb
\frac{\alpha(\eta,q,u)}{\beta(\eta,q,u)}=-\frac{\alpha(q,\eta,-u)}{\beta(\eta,q,-u)}
\ee
Finally, the cancelation of $K^{(2)}_{12}$ leads also to the same Bethe equation \eqref{Bethe1} thanks to
the following identity:
\eb
0&=&\left(\frac{\de_{1u}(-q)}{\ga_{1u}(-q,q)y_{12}(q-2\eta)}+
\frac{z_{1u}(q)\al_{2u}(\eta,q)\omega_{u1}}{\be_{2u}(\eta,q)y_{u1}(q)}\right)\times \frac{\tf(q-3\et)}{\tf(q-\et)}+\\
&&\left(\frac{\de_{u1}(q)}{\ga_{u1}(-q,q)y_{12}(q-2\eta)}-\frac{\al_{u1}(q,\eta)}{\ga_{u1}(-q,q)y_{u2}(q-2\eta)}
\right)\times \frac{\tf(q-3\et)}{\tf(q-\et)}+\\
&&\frac{1}{y_{u1}(q)}\left(\frac{z_{u1}(q)\al_{u2}(q-2\eta,\eta)}{\be_{u2}(\eta,q-2\eta)}-\frac{\al_{u1}(q,\eta)\al_{12}(q-2\eta,\eta)}
{\be_{u1}(\eta,q,)\be_{12}(\eta,q-2\eta)}\right)\times \\
&&\sqrt{\frac{\tf(q-\eta)\tf(q-5\et)}{\tf(q+\et)\tf(q-3\et)}}\frac{\tf(u_{12}-2\et)\tf(u_{12}+1/2+\et)}
{\tf(u_{12}+2\et)\tf(u_{12}+1/2-\et)}+\\
&&\frac{\al_{u1}(q,\eta)\al_{12}(q+2\eta,\eta)}{\be_{u1}(\eta,q)y_{u1}(q)\be_{12}(\et,q-2\eta)}\times
\sqrt{\frac{\tf(q-\eta)\tf(q-5\et)}{\tf(q+\et)\tf(q-3\et)}}\frac{\tf(u_{21}-2\et)\tf(u_{21}+1/2+\et)}
{\tf(u_{21}+2\et)\tf(u_{21}+1/2-\et)}
+\\
&&\frac{\al_{u1}(q,\eta)}{\be_{u1}(-q,\eta)}\left(\frac{z_{u1}(q)}{\omega_{u1}y_{u2}(q)}-
\frac{\al_{u1}(\eta,-q)}{\be_{u1}(\eta,q)y_{12}(q)}\right)\times \\
&&\sqrt{\frac{\tf(q+3\et)\tf(q-3\et)\tf(q-\et)}
{\tf(q+\et)^3}}
\frac{\tf(u_{21}-2\et)\tf(u_{21}+1/2+\et)}
{\tf(u_{21}+2\et)\tf(u_{21}+1/2-\et)}
\ee
Now it remains to fix $f(q)$ so as to ensure that the Bethe equation (hence its solutions) do not depend on $q$. This
can be achieved by choosing
\eb
f(q)&=&e^{cq}\frac{\tf(q-\et)^{\frac{n}{2}}}{\tf(q+\et)^{\frac{n}{2}}}
\ee
where $c$ is an arbitrary constant.

The simultaneous vanishing of $K^{(1)}_j$, $K^{(3)}_j$ and $K^{(2)}_{jl}$ is ensured by the same condition on the
spectral parameters:
\eb
\prod_{k=1}^{n}\frac{\tf(u_j-z_k+2\et)}{\tf(u_j-z_k)}&=&e^{2c\et}\prod_{\substack{k=1 \\ k\neq j}}^n
\frac{\tf(u_{jk}-2\et)\tf(u_{jk}+1/2+\et)}{\tf(u_{jk}+2\et)\tf(u_{jk}+1/2-\et)}
\ee

Assuming a set of solutions $\{u_1,\ldots,u_n\}$ to this Bethe equation is known we write the eigenvalues
of the transfer matrix as:
\eb
\Lambda(u,\{u_i\})&=&e^{-2\eta c}\prod_{k=1}^n \frac{\tf(u_k-u-2\eta)\tf(3\eta+1/2-u+z_k)\tf(u-z_k+2\eta)}
{\tf(u_k-u)\tf(3\eta+1/2)\tf(-2\eta)}+\\
&&\prod_{k=1}^n\frac{\tf(3\eta+1/2-u+z_k)\tf(u-z_k)}{\tf(-2\eta)\tf(3\eta+1/2)}+\\
&&e^{2\eta c}\prod_{k=1}^n\frac{\tf(3\eta+1/2-u+u_k)\tf(u-z_k)\tf(\eta+1/2-u+z_k)}{\tf(\eta+1/2-u+u_k)\tf(3\eta+1/2)\tf(-2\eta)}
\ee

\section{Conclusions}
We showed in this paper that the algebraic Bethe ansatz method can be implemented in the
elliptic quantum group $E_{\tau,\eta}(A_2^{(2)})$. This elliptic quantum group
is another example of the algebras associated rank one classical Lie algebras.
We defined the Bethe state creation operators
through a recurrence relation having the same structure as the ones in \cite{MaNa2,Ta}. As an example
we took the transfer matrix associated to the tensor product of fundamental representations and wrote the corresponding Bethe equations and eigenvalues.

\appendix

\subsection*{Acknowledgements}

This work was supported by the project POCI/MAT/58452/2004,
in addition to that Z. Nagy benefited from
the FCT grant SFRH/BPD/25310/2005. N. Manojlovi\'c acknowledges additional support from
SFRH/BSAB/619/2006. The authors also wish to thank Petr Petrovich Kulish for kind interest and encouragement.

\end{document}